\documentclass[12pt]{article}      

\pdfoutput=1

\advance\voffset by -0.25cm

\usepackage{graphicx,amsmath,amsfonts,amsthm,bbm,setspace,url}
\usepackage{natbib,verbatim,mathrsfs}

\newtheorem{theorem}{Theorem}
\newtheorem{proposition}[theorem]{Proposition}
\newtheorem{lemma}[theorem]{Lemma}

\begin{document}
\sloppy
\baselineskip=20pt

\newcommand{\Cov}{{\rm Cov}}
\newcommand{\Cor}{{\rm Cor}}
\newcommand{\Var}{{\rm Var}}
\newcommand{\real}{\mathbb{R}}
\newcommand{\complex}{\mathbb{C}}
\newcommand{\indicator}{\mathbbm{1}}
\newcommand{\E}{\mathbb{E}}
\def\T{^{\rm T}}

\newcommand{\ba}{\mathbf{a}}
\newcommand{\bb}{\mathbf{b}}
\newcommand{\bc}{\mathbf{c}}
\newcommand{\bd}{\mathbf{d}}
\newcommand{\be}{\mathbf{e}}
\newcommand{\beff}{\mathbf{f}}
\newcommand{\bg}{\mathbf{g}}
\newcommand{\bh}{\mathbf{h}}
\newcommand{\bi}{\mathbf{i}}
\newcommand{\bj}{\mathbf{j}}
\newcommand{\bk}{\mathbf{k}}
\newcommand{\bl}{\mathbf{l}}
\newcommand{\bell}{\boldsymbol{\ell}}
\newcommand{\bm}{\mathbf{m}}
\newcommand{\bn}{\mathbf{n}}
\newcommand{\bo}{\mathbf{o}}
\newcommand{\bp}{\mathbf{p}}
\newcommand{\bq}{\mathbf{q}}
\newcommand{\br}{\mathbf{r}}
\newcommand{\bs}{\mathbf{s}}
\newcommand{\bt}{\mathbf{t}}
\newcommand{\bu}{\mathbf{u}}
\newcommand{\bv}{\mathbf{v}}
\newcommand{\bw}{\mathbf{w}}
\newcommand{\bx}{\mathbf{x}}
\newcommand{\by}{\mathbf{y}}
\newcommand{\bz}{\mathbf{z}}

\newcommand{\vep}{\varepsilon}
\newcommand{\vphi}{\varphi}

\newcommand{\sI}{\mathscr{I}}
\newcommand{\sR}{\mathscr{R}}
\newcommand{\sP}{\mathscr{P}}

\newcommand{\bB}{\mathbf{B}}
\newcommand{\bC}{\mathbf{C}}
\newcommand{\bH}{\mathbf{H}}
\newcommand{\bI}{\mathbf{I}}
\newcommand{\bL}{\mathbf{L}}
\newcommand{\bM}{\mathbf{M}}
\newcommand{\bR}{\mathbf{R}}
\newcommand{\bU}{\mathbf{U}}
\newcommand{\bW}{\mathbf{W}}
\newcommand{\bX}{\mathbf{X}}
\newcommand{\bY}{\mathbf{Y}}
\newcommand{\bZ}{\mathbf{Z}}
\newcommand{\bomega}{\boldsymbol{\omega}}
\newcommand{\bbeta}{\boldsymbol{\beta}}

\newcommand{\bphi}{\boldsymbol{\phi}}
\newcommand{\blambda}{\boldsymbol{\lambda}}
\newcommand{\btheta}{\boldsymbol{\theta}}
\newcommand{\bvep}{\boldsymbol{\varepsilon}}
\newcommand{\bmu}{\boldsymbol{\mu}}

\begin{center}{
  \Large \bf
  Coherence for Random Fields
}

\bigskip
\bigskip

{\bf William Kleiber\footnote[1]{Department of Applied Mathematics,
  University of Colorado, Boulder, CO.  
  Author e-mail: \texttt{william.kleiber@colorado.edu}}
}

\bigskip
\bigskip

{\bf \today}

\end{center}

\bigskip
\bigskip

\begin{abstract}
Multivariate spatial field data are increasingly common and whose modeling typically 
relies on building cross-covariance functions to describe cross-process 
relationships.  An alternative viewpoint is to model the matrix of spectral 
measures.  We develop the notions of coherence, phase and gain for multidimensional 
stationary processes.  Coherence, as a function of frequency, 
can be seen to be a measure of linear relationship between two spatial processes 
at that frequency band.  We use the coherence function to illustrate fundamental 
limitations on a number of previously proposed constructions for multivariate 
processes, suggesting these options are not viable for real data.  
We also give natural interpretations to cross-covariance parameters of the 
Mat\'ern class, where the smoothness indexes dependence at low frequencies while the 
range parameter can imply dependence at low or high frequencies.  
Estimation follows from smoothed multivariate periodogram matrices.  
We illustrate the estimation and interpretation of these functions on two 
datasets, forecast and reanalysis sea level pressure and geopotential 
heights over the equatorial region.  Examining these functions lends insight that 
would otherwise be difficult to detect and model using standard 
cross-covariance formulations.

\end{abstract}

\section{Introduction}  \label{sec:intro}

The theory of univariate continuous stochastic processes has become well 
developed over nearly a century of research.  The past quarter century 
or so has seen an increasing interest and development of models for 
multivariate spatial processes.  The recent review by \citet{genton2015} 
gives a relatively comprehensive treatment of the basic approaches that have 
been explored to build stochastic spatial models.  In the 
discussion, \citet{bevilacqua2015} pose the question, given the recent deluge 
of multivariate constructions, ``which parametric model is more flexible?''
Indeed, the relative strengths and weaknesses of multivariate 
models have been explored only empirically, that is, by testing a battery 
of different models on particular datasets, and comparing performance 
either by likelihood values or by predictive cross-validation (in the 
multivariate context, spatial prediction is known as co-kriging).  Thus, 
a fundamental open question is: (1) to what extent can the flexibilities of 
model constructions be compared theoretically?  Additionally, 
most models are motivated in the covariance domain, and the 
natural follow-up question is: (2) are there other approaches than covariance 
to measure and quantify spatial dependence?  

We introduce the notion of spectral coherence, phase and gain for multidimensional 
and multivariate spatial random fields.  We propose that these functions allow 
for natural partial answers to the critical open questions (1) and (2).  
We show that a number of previously proposed models lack sufficient practical 
flexibility in terms of prediction, and we suggest insights into well established 
models such as the multivariate Mat\'ern, 
where parameters such as the cross-covariance smoothness and range have 
had elusive direct interpretations that relate to process behavior.  

Let us illustrate the ideas developed in this manuscript by considering the time 
series case first.  Suppose $Z_i(t), i=1,2$ is a bivariate complex-valued 
weakly stationary time series on $t\in\real$ with covariance functions 
$\Cov(Z_k(t+h),Z_k(t)) = C_{kk}(h)$ and cross-covariance functions 
$\Cov(Z_k(t+h),Z_\ell(t)) = C_{k\ell}(t)$ for $k\not=\ell$.  The corresponding 
spectral densities are $f_{k\ell}(\omega) = (2\pi)^{-1}\int_\real C_{k\ell}(h) 
\exp(-i \omega h) {\rm d}h$ for $\omega\in\real$.  Spectral modeling in time 
series is well developed; for example, traditional autoregressive moving average 
models imply processes with rational spectral densities.  
The function 
\[
  \gamma(\omega)^2 = \frac{|f_{12}(\omega)|^2}{f_{11}(\omega)f_{22}(\omega)}
\]
is known as the squared coherence function, and can be interpreted as a 
quantification of the linear relationship between $Z_1(t)$ and $Z_2(t)$ at 
frequency $\omega$ \citep{brockwell2009}.

We entertain two example datasets from the atmospheric sciences.  Both are 
reforecast and reanalysis data products over the equatorial region based on a 
well established numerical weather prediction (NWP) model.  
Reanalysis forecasts are from a fixed version of a 
NWP model that are run retrospectively to generate 
a large database of model forecasts and analyses (in this context, an analysis 
can be considered a best estimate of the current state of the atmosphere).  
First we look at forecasted surfaces of sea level pressure at daily forecast 
horizons between 24 and 192 hours.  We show that coherence can be used as 
a diagnostic to assess forecast quality, and additionally illustrate 
frequency bands at which forecasts improve over time.
The second dataset involves geopotential heights at differing pressure levels.  
We show that coherence and phase extract and highlight qualities of the 
spatial relationship between different pressure levels that are difficult 
to model using extant multivariate covariance constructions, and indeed 
illustrate some fundamental limitations of existing popular constructions.

\section{Spectra for Multivariate Random Fields}  \label{sec:spectra}

Suppose $\bZ(\bs) = (Z_1(\bs),\ldots,Z_p(\bs))\T\in\complex^p$ is a $p$-variate weakly 
stationary random field on $\bs\in\real^d$ admitting a matrix-valued covariance 
function $\bC(\bh) = (C_{ij}(\bh))_{i,j=1}^p$ where $C_{ij}(\bh) = 
\Cov(Z_i(\bs + \bh),Z_j(\bs))$.  For simplicity of exposition we suppose 
$\bZ(\bs)$ is a mean zero process.  For complex-valued stationary processes, 
$\Cov(Z_i(\bs_1),Z_j(\bs_2)) = \E(Z_i(\bs_1)\overline{Z_j(\bs_2)})$, so that 
$C_{ij}(\bh) = \overline{C_{ji}(\bh)}$.  
The main obstacle to multivariate process modeling is developing flexible 
classes of matrix-valued covariance functions $\bC$ that are nonnegative 
definite.  We say $\bC$ is nonnegative definite if, for any choices 
of $a_{ik}\in\complex$ and locations $\bs_k\in\real^d$ for $i=1,\ldots,p$ and 
$k=1,\ldots,n$ we have 
\[
  \sum_{i=1}^p \sum_{j=1}^p \sum_{k=1}^n \sum_{\ell=1}^n 
  a_{ik} \overline{a_{j\ell}} C_{ij}(\bs_k,\bs_\ell) \geq 0.
\]
Note this reduces to the usual definition of nonnegative definiteness 
for a univariate covariance, $p=1$.

For univariate processes, Bochner's Theorem states that $C_{ii}(\bh)$ is a 
valid (i.e., nonnegative definite) function if and only if it can be written 
\[
  C(\bh) = \int_{\real^d} \exp(i\bomega\T \bh) {\rm d}F(\bomega)
\]
where $F$ is a positive finite measure \citep{stein1999}.  
If $F$ admits a density $f$ with respect to the Lebesgue measure on $\real^d$, 
we call it the spectral density for $C$.  
The multivariate extension of Bochner's fundamental result is given by 
\citet{cramer1940}, and is contained in the following theorem specialized to 
covariances admitting spectral densities.
\begin{theorem}[Cram\'er 1940]  \label{th:cramer}
  A matrix-valued function $\bC:\real^d\to\complex^{p\times p}, 
  \bC = (C_{ij})_{i,j=1}^p$ is nonnegative definite if and only if 
  \begin{align*}
    C_{ij}(\bh) = \int_{\real^d} \exp(i\bomega\T\bh) f_{ij}(\bomega) {\rm d}\bomega
  \end{align*}
  for $i,j=1,\ldots,p$ such that the matrix 
  $\beff(\bomega) = (f_{ij}(\bomega))_{i,j=1}^p$ is nonnegative definite  
  for all $\bomega\in\real^d$.
\end{theorem}
The functions $f_{ij}(\bomega)$ are the spectral and cross-spectral 
densities for the marginal and cross-covariance functions $C_{ij}(\bh)$.  
Note that $f_{ij}(\bomega) = \overline{f_{ji}(\bomega)}$.  When the spectral 
density exists, it can be solved for as the Fourier transform of the covariance 
function,
\begin{align*}
  f_{ij}(\bomega) = \frac{1}{(2\pi)^d} \int_{\real^d} 
  \exp(-i\bomega\T\bh) C_{ij}(\bh) {\rm d}\bh.
\end{align*}
Theorem \ref{th:cramer} has primarily been used in practice to build multivariate 
covariance models, by specifying matrices of spectral densities that are 
nonnegative definite for all frequencies.

\subsection{Coherence}  \label{subsec:coh}

In time series, the notion of frequency coherence is well developed, and 
can be used, for instance, to assess whether 
one time series is related to another by a time invariant linear filter.  These 
notions carry over to the spatial case, and form the point of entry for our 
analyses.

If $C_{ij}(\bh), i,j=1,2$ form a matrix-valued covariance function with 
associated spectral densities $f_{ij}(\bomega)$, then define the coherence 
function (or coherency function)
\begin{align*}
  \gamma(\bomega) = \frac{f_{12}(\bomega)}{\sqrt{f_{11}(\bomega)f_{22}(\bomega)}}.
\end{align*}
We might assume $f_{ii}(\bomega) > 0$ for all $\bomega\in\real^d$ for $i=1,2$, 
but can define $\gamma(\bomega) = 0$ if $f_{ii}(\bomega) = 0$.  The coherence 
function can be complex-valued, so in practice we examine the absolute coherence 
function, $|\gamma(\bomega)|$.  
The real-valued function $|\gamma(\bomega)|^2$ is the squared coherence function, 
and by Theorem \ref{th:cramer}, $0\leq |\gamma(\bomega)|^2 \leq 1$ for all 
$\bomega$.  Values of $|\gamma(\bomega)|$ near unity indicate a linear relationship 
between $Z_1(\bs)$ and $Z_2(\bs)$ at particular frequency bands.

The following theorem relates the coherence to optimal prediction of a 
random process based on another process.  The predictive estimator is based on 
a kernel smoothed process which is a natural predictor given the interpretation 
of the univariate kriging weights as a kernel function \citep{kleiber2015}.
\begin{theorem}  \label{th:prediction}
  Suppose $(Z_1(\bs),Z_2(\bs))\T$ is a complex-valued mean zero weakly stationary 
  bivariate field with matrix-valued covariance $\bC(\bh)$ admitting 
  a spectral density matrix $\beff(\bomega) = (f_{ij}(\bomega))_{i,j=1}^2$ 
  that is everywhere nonzero.  
  Then the continuous square integrable function $K(\bu):\real^d\to\complex$ that 
  minimizes $\E\big|Z_1(\bs_0) - \int_{\real^d} K(\bu - \bs_0) Z_2(\bu) 
  {\rm d}\bu\big|^2$ is 
  \begin{align}  \label{eq:k.pred.coh}
    K(\bu) = \frac{1}{(2\pi)^d}\int_{\real^d} \exp(-i\bomega\T\bu)
    \frac{f_{12}(\bomega)}{f_{22}(\bomega)} {\rm d}\bomega
    = \frac{1}{(2\pi)^d}\int_{\real^d} \exp(-i\bomega\T\bu)
    \sqrt{\frac{f_{11}(\bomega)}{f_{22}(\bomega)}}\gamma(\bomega) {\rm d}\bomega.
  \end{align}
  Additionally, the spectral density of the predictor 
  $\hat{Z}_1(\bs_0) = \int_{\real^d} K(\bu - \bs_0) Z_2(\bu) {\rm d}\bu$ is 
  \begin{align}  \label{eq:cond.coh}
    f_{1|2}(\bomega) = f_{11}(\bomega)|\gamma(\bomega)|^2
  \end{align}
  for all $\bomega\in\real^d$.
\end{theorem}
In particular, the relationship (\ref{eq:k.pred.coh}) implies that the optimal 
weighting function is modulated by the coherence between the two processes, 
and indeed has greater spectral weight on frequencies with high coherence.  
An immediate corollary to Theorem \ref{th:prediction} is
\[
  |\gamma(\bomega)|^2 = \frac{f_{1|2}(\bomega)}{f_{11}(\bomega)}.
\]
Thus, the coherence has an attractive interpretation as the amount of 
variability that can be attributed to a linear relationship between two 
processes at a particular frequency.  
In the following development, we use the coherence function to illuminate 
fundamental limitations on some popular multivariate covariance constructions.


The coherence function can be used as a tool to compare proposed multivariate 
models, as an indicator of the amount of flexibility of bivariate relationships 
at differing frequencies.  For example, a rather classic approach to specifying 
covariances is separability, setting $\bC(\bh) = \bR C(\bh)$ where $C(\bh)$ is a 
univariate covariance function and $\bR$ is a $p\times p$ positive definite 
matrix \citep{mardia1993,helterbrand1994,bhat2010}.  This approach has been 
empirically shown to be insufficiently flexible, and the following lemma contributes 
to the empirical results.
\begin{lemma}  \label{lem:separability}
  If $\bC(\bh) = \bR C(\bh)$ where $C:\real^d \to \real$ is a covariance 
  function and $\bR$ is a $p\times p$ positive definite matrix with $(i,j)$th 
  entry $r_{ij}$, then the squared coherence between the $i$th and $j$th process 
  is constant, in particular $\gamma_{ij}(\bomega)^2 = (r_{ij}r_{ji})/(r_{ii}r_{jj})$.
\end{lemma}

A more sophisticated method of generating multivariate covariance structures is 
to convolve univariate square integrable functions 
\citep{gaspari1999,oliver2003,gaspari2006,majumdar2007}.  
In particular, if $c_i:\real^d\to\real$ are square integrable functions for 
$i=1,\ldots,p$ then $C_{ij}(\bh) = (c_i \star c_j)(\bh)$ is a valid matrix-covariance 
function where $\star$ denotes the convolution operator.  This is 
sometimes known as covariance convolution (especially when $c_i$ are positive 
definite functions to begin with).  
The following proposition suggests this approach 
to model building is overly-restrictive, and indeed implies that the resulting 
coherence is necessarily constant over all frequencies.
\begin{proposition}  \label{prop:c.conv}
  If $c_1$ and $c_2$ are square integrable functions on $\real^d$ 
  and a matrix-valued covariance is defined via 
  $C_{ij} = c_i \star c_j$ for $i,j=1,2$ where $\star$ denotes convolution, 
  then $\gamma(\bomega) \equiv 1$ for all $\bomega\in\real^d$ such that 
  the Fourier transforms of $c_i$ and $c_j$ are nonzero.
\end{proposition}

Multivariate processes can sometimes be modeled as being related by local 
averaging.  For example, the relationship between column integrated ozone observations 
and local ozone might be plausibly modeled as observations being locally 
averaged over the true underlying field \citep{cressie2008}.  Wind observations 
are often time averaged over moving windows to produce smoother and more 
stable observation series \citep{hering2015}.  The following proposition 
characterizes the coherence in such situations, and serves to illustrate the 
intimate link between process relationship and coherence.
\begin{proposition}  \label{prop:z.conv}
  If $Z_1(\bs)$ is a weakly stationary stochastic process 
  and $Z_2(\bs) = \int_{\real^d} K(\bu - \bs) Z_1(\bu) {\rm d}\bu$ 
  for some continuous square integrable kernel function $K:\real^d\to\real$ that 
  is symmetric, then $\gamma(\bomega) \equiv 1$ for all $\bomega\in\real^d$ such 
  that the Fourier transform of $K$ is nonzero.
\end{proposition}
According to Proposition \ref{prop:z.conv}, estimated coherences near unity over 
all frequency bands may be indicative of a linear or local averaged 
relationship between processes, and this result may serve as the theoretical 
basis for testing such a hypothesis.  \citet{fuentes2006} uses a similar notion 
to develop a test for separability of space-time processes.

The kernel convolution method, introduced by 
\citet{verhoef1998} and \citet{verhoef2004}, originally involved representing 
a process as a moving average against a white noise process.  In simple cases 
this yields the covariance convolution model.  This can be generalized to 
\begin{align}  \label{eq:k.convolution}
  Z_k(\bs) = \int_{\real^d} g_k(\bx - \bs) W(\bx) {\rm d}\bx
\end{align}
where $W(\bx)$ is a mean zero stationary process with covariance $C(\bh)$, and 
$g_k:\real^d\to\real$ are square integrable symmetric 
kernel functions for $k=1,\ldots,p$.  As with each previous construction, this 
approach also yields constant coherence.
\begin{proposition}  \label{prop:k.conv}
  If $Z_1(\bs)$ and $Z_2(\bs)$ are constructed as in (\ref{eq:k.convolution}), 
  then $\gamma(\bomega)$ is constant for all $\bomega\in\real^d$.
\end{proposition}

For all of these models, separable, covariance convolution and kernel convolution, 
the resulting multivariate structure is restricted to constant coherence.  In light of 
Theorem \ref{th:prediction}, this suggests that {\em none} of these models can attain 
optimal prediction for any multivariate processes exhibiting nontrivial coherences.  
Indeed the examples below in Section \ref{sec:example} exhibit nonconstant 
coherence, and call for more flexible modeling frameworks.

\subsection{Phase and Gain}  \label{subsec:phase.gain}

Similar notions to frequency coherence can be motivated by examining 
spectral density matrices.  If $(Z_1(\bs),Z_2(\bs))\T$ is a stationary random 
vector with spectral density matrix $(f_{ij}(\bomega))_{i,j=1}^2$ then define 
$A(\bomega) = f_{12}(\bomega) / f_{11}(\bomega)$.  Note that $A(\bomega)$ is 
possibly complex-valued.  We define the gain function $G(\bomega) = 
|A(\bomega)|$ which is sometimes referred to as the gain of $Z_2(\bs)$ on 
$Z_1(\bs)$ in time series \citep{brockwell2009}.  Additionally define the 
phase function at frequency $\bomega$ as $\phi(\bomega) = \arg A(\bomega)$, the 
complex argument of $A(\bomega)$.  Note that the phase function satisfies 
$\phi(\bomega)\in(-\pi,\pi]$ and $\phi(-\bomega) = -\phi(\bomega)$.  

The interpretations of gain and phase are most clear when considering 
processes built by the relationship $Z_1(\bs) = \alpha Z_2(\bs - \bu)$ for 
some $\bu\in\real^d$ and $\alpha \not= 0$, 
i.e., $Z_1$ is a shifted and rescaled version of $Z_2$.  
Then it is straightforward to show that the phase function is 
\begin{align*}
  \phi(\bomega) = 
  \begin{cases}
    -\bomega\T\bu\,(\hspace{-4.2mm}\mod 2\pi), & \alpha > 0\\
    \pi-\bomega\T\bu\,(\hspace{-4.2mm}\mod 2\pi) & \alpha < 0.
  \end{cases}
\end{align*}
\citet{li2011} develop an approach to modeling this type of asymmetric 
cross-covariance behavior.  This result shows that their construction will 
have phase shift function that depends on the angle 
$\bomega\T\bu$.  This result may be used as an 
exploratory data approach or as the basis for a statistical test to suggest 
whether a pair of spatial processes exhibit an asymmetric relationship; 
\citet{li2011} use the empirical cross-correlation function to visually 
assess such asymmetric behavior.  The gain function in this case 
is simply $G(\bomega) = |\alpha|$; all frequency components of $Z_2$ are 
exaggerated by an amount $\alpha$ for $Z_1$.

Below we consider some multivariate constructions that are particular to 
real-valued processes having real-valued spectral matrices.  Any model with 
real-valued cross-spectral density has $\phi(\bomega)\equiv 0$, but a 
possibly non-trivial gain function.  Thus, testing for $\phi(\bomega)\equiv 0$ 
can be viewed as a test for a real-valued cross-spectral density, which 
seems very relevant given most multivariate models are developed under this 
assumption.

\subsection{Revisiting the Multivariate Mat\'ern}  \label{subsec:mm}

The multivariate Mat\'ern is a model for matrix-valued covariance functions 
such that each process is marginally described by a Mat\'ern covariance 
function, and all cross-covariance functions also fall in the Mat\'ern class 
\citep{gneiting2010,apanasovich2012}.  Specifically, the multivariate Mat\'ern 
imposes $C_{ii}(\bh) = \sigma_i^2 {\rm M}(\bh\,|\,\nu_i,a_i)$ for $i=j$ and 
$C_{ij}(\bh) = \rho_{ij} \sigma_i\sigma_j {\rm M}(\bh\,|\,\nu_{ij},a_{ij})$ 
for $1\leq i\not=j\leq p$.  Here, 
${\rm M}(\bh\,|\,\nu,a) = (2^{1-\nu}/\Gamma(\nu)) (a\|\bh\|)^\nu 
{\rm K}_\nu(a\|\bh\|)$ where $K_{\nu}$ is a modified Bessel function of the 
second kind of order $\nu$.  \citet{gneiting2010} and \citet{apanasovich2012} 
discuss restrictions on the parameters $\nu_{i}, a_i, \nu_{ij}, a_{ij}$ and $\rho_{ij}$ 
that result in a valid model.  

The Mat\'ern class is popular due to the smoothness parameters $\nu > 0$ which 
continuously index smoothnesses of the sample paths of the process.  
In particular, sample paths are $m$ times differentiable if and only if $\nu > m$, 
and there is an additional relationship between $\nu$ and the fractal dimension 
in that sample paths have dimension $\max(d,d+1-\nu)$ \citep{goff1988,handcock1993}.  
These interpretations and implications also hold in the multivariate case, 
where $\nu_i$ indexes the smoothness of the $i$th component process $Z_i(\bs)$.  
The parameters $a_i$ act as range parameters, and control the rate of 
decay of spatial correlation away from the origin.  

A standing issue with the multivariate Mat\'ern is that the cross-covariance 
parameters, $\nu_{ij}$ and $a_{ij}$ for $i\not=j$, do not have straightforward 
interpretations that are analogous to the marginal smoothness and range 
interpretations, and indeed nowhere in the literature have these parameters been 
linked directly to process behavior.  We find that these parameters have 
direct interpretations when considering the coherence function between two 
processes.

The coherence function for a bivariate process with multivariate Mat\'ern 
correlation is 
\begin{align*}
  \gamma(\bomega)^2 = \rho^2 \frac{\Gamma(\nu_{12}+d/2)^2 \Gamma(\nu_1)\Gamma(\nu_2)}
  {\Gamma(\nu_1+d/2) \Gamma(\nu_2+d/2) \Gamma(\nu_{12})^2} 
  \frac{a_{12}^{4\nu_{12}}}{a_1^{2\nu_1}a_2^{2\nu_2}} 
  \frac{(a_1^2 + \|\bomega\|^2)^{\nu_1+d/2} (a_2^2 + \|\bomega\|^2)^{\nu_2+d/2}}
  {(a_{12}^2 + \|\bomega\|^2)^{2\nu_{12}+d}}.
\end{align*}
We explore in detail two simplified versions of this coherency function.  
First, consider the case where $a_1 = a_2 = a_{12} = a$, all covariance 
and cross-covariance functions share a common range.  Then 
\begin{align}  \label{eq:coh.mm.nu}
  \gamma(\bomega)^2 = \rho^2 \frac{\Gamma(\nu_{12}+d/2)^2 \Gamma(\nu_1)\Gamma(\nu_2)}
  {\Gamma(\nu_1+d/2) \Gamma(\nu_2+d/2) \Gamma(\nu_{12})^2} 
  (a + \|\bomega\|^2)^{\nu_1 + \nu_2 - 2 \nu_{12}}.
\end{align}

\begin{figure}[t]
  \centering
  \includegraphics[width=\linewidth]{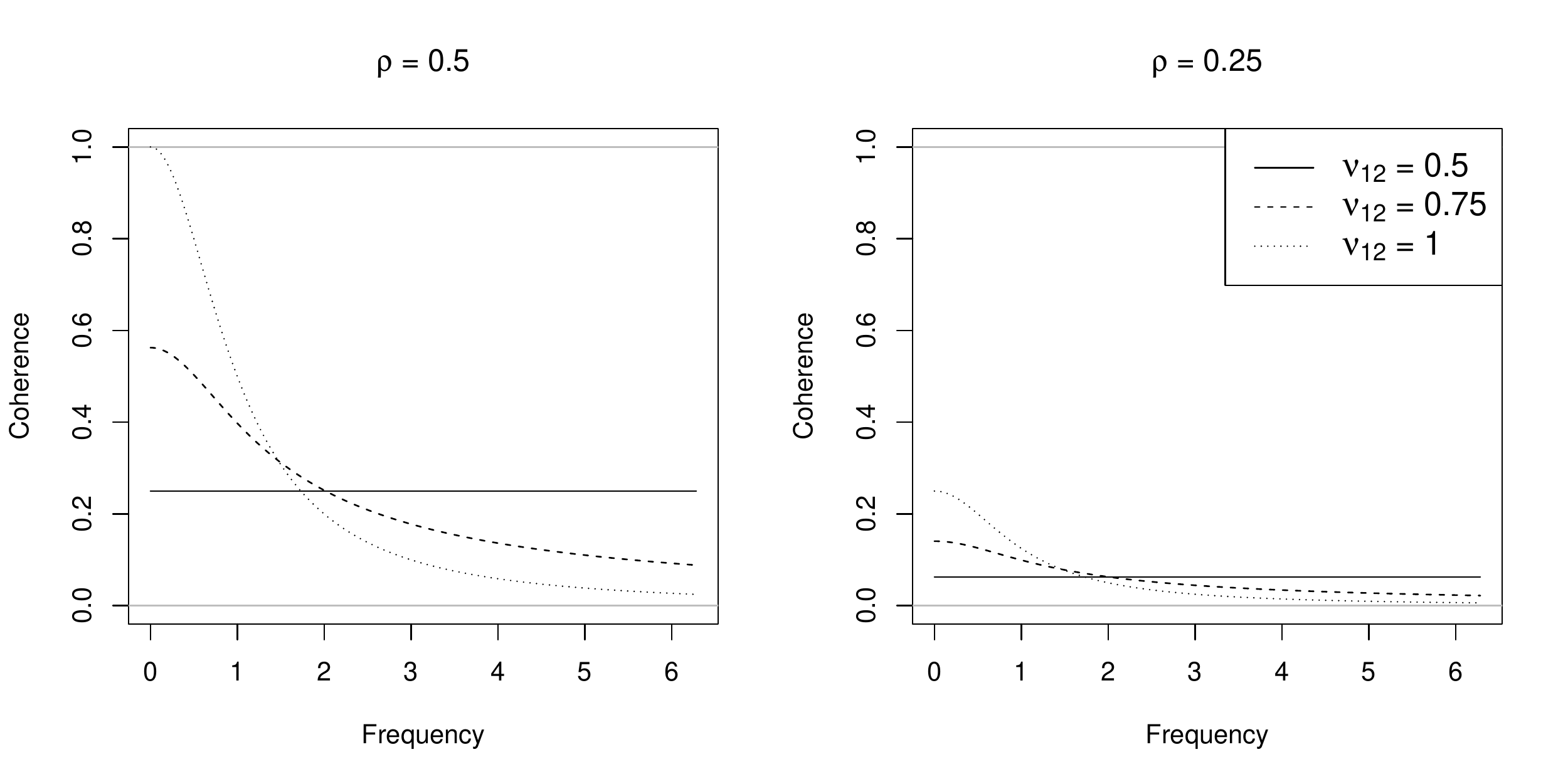}
  \caption{Coherence functions for various bivariate Mat\'erns with $a_{11} = a_{12} = 
  a_{22} = 1, \nu_{11} = \nu_{22} = 0.5$, varying $\nu_{12}$ and $\rho$.
  \label{fig:coh.mm.nu}}
\end{figure}

The multivariate Mat\'ern is a valid model only if $\nu_{12} \geq (\nu_1 + \nu_2)/2$, 
and thus by inspecting (\ref{eq:coh.mm.nu}) we note some modeling implications 
resulting from these restrictions.  
First, if $\nu_{12} = (\nu_1 + \nu_2)/2$, 
the coherency is constant across all frequencies, implying a constant 
linear relationship between two processes at all frequency bands.  Second, 
if $\nu_{12} > (\nu_1 + \nu_2)/2$, we have greater coherency at low frequencies, 
with $\gamma(\bomega) \to 0$ as $\|\bomega\|\to\infty$.  
This analysis seems to suggest a natural interpretation of the cross-covariance 
smoothness in that it controls the amount of cross-process dependence at 
various frequencies, but can only imply greater or equal coherence at 
low frequencies versus high frequencies.  Figure \ref{fig:coh.mm.nu} serves 
to illustrate these points, showing various coherence functions for 
common length scale parameters $a_{11} = a_{12} = a_{22} = 1$, varying 
$\nu_{12}$ and $\rho$.  

Perhaps surprisingly, a similar analysis suggests the cross-covariance 
range parameter $a_{12}$ induces potentially greater flexibility in the 
coherence function than the smoothness parameter $\nu_{12}$.  In particular, 
if $\nu_1 = \nu_2 = \nu_{12} = \nu$, the coherence function is 
\begin{align}  \label{eq:coh.mm.a}
  \gamma(\bomega)^2 = \rho^2 
  \left(\frac{a_{12}^2}{a_1 a_2}\right)^{2\nu}
  \left(\frac{(a_1^2 + \|\bomega\|^2) (a_2^2 + \|\bomega\|^2)}
  {(a_{12}^2 + \|\bomega\|^2)^{2}}\right)^{\nu + d/2}.
\end{align}
Examining (\ref{eq:coh.mm.a}), we see that, depending on whether $a_{12} \leq 
\min(a_1,a_2)$ or $a_{12} \geq \max(a_1,a_2)$ the coherence will have 
distinct behavior.  For simplicity, set $a_1 = a_2 = a$.  
If $a_{12} < a$, the coherence will be greater for small frequencies than 
high frequencies, similar to the behavior implied by (\ref{eq:coh.mm.nu}) 
when $\nu_{12} > (\nu_1 + \nu_2)/2$.  However, here $\gamma(\bomega) \to 
\rho a_{12}^{2\nu} / (a_1 a_2)^\nu$ 
as $\|\bomega\|\to\infty$, implying non-negligible coherence between 
processes at high frequencies, unlike that in (\ref{eq:coh.mm.nu}).  
If $a_{12} > a$, we have the complementary result that 
$\gamma(\bomega_1) < \gamma(\bomega_2)$ for $\|\bomega_1\| < \|\bomega_2\|$, 
that is, two processes share behavior at high frequencies rather than low.  
Figure \ref{fig:coh.mm.a} illustrates these scenarios, and seems to 
suggest that, at least as far as coherence is concerned, the cross-covariance 
range parameter yields potentially greater flexibility than the cross-smoothness.

\begin{figure}[t]
  \centering
  \includegraphics[width=\linewidth]{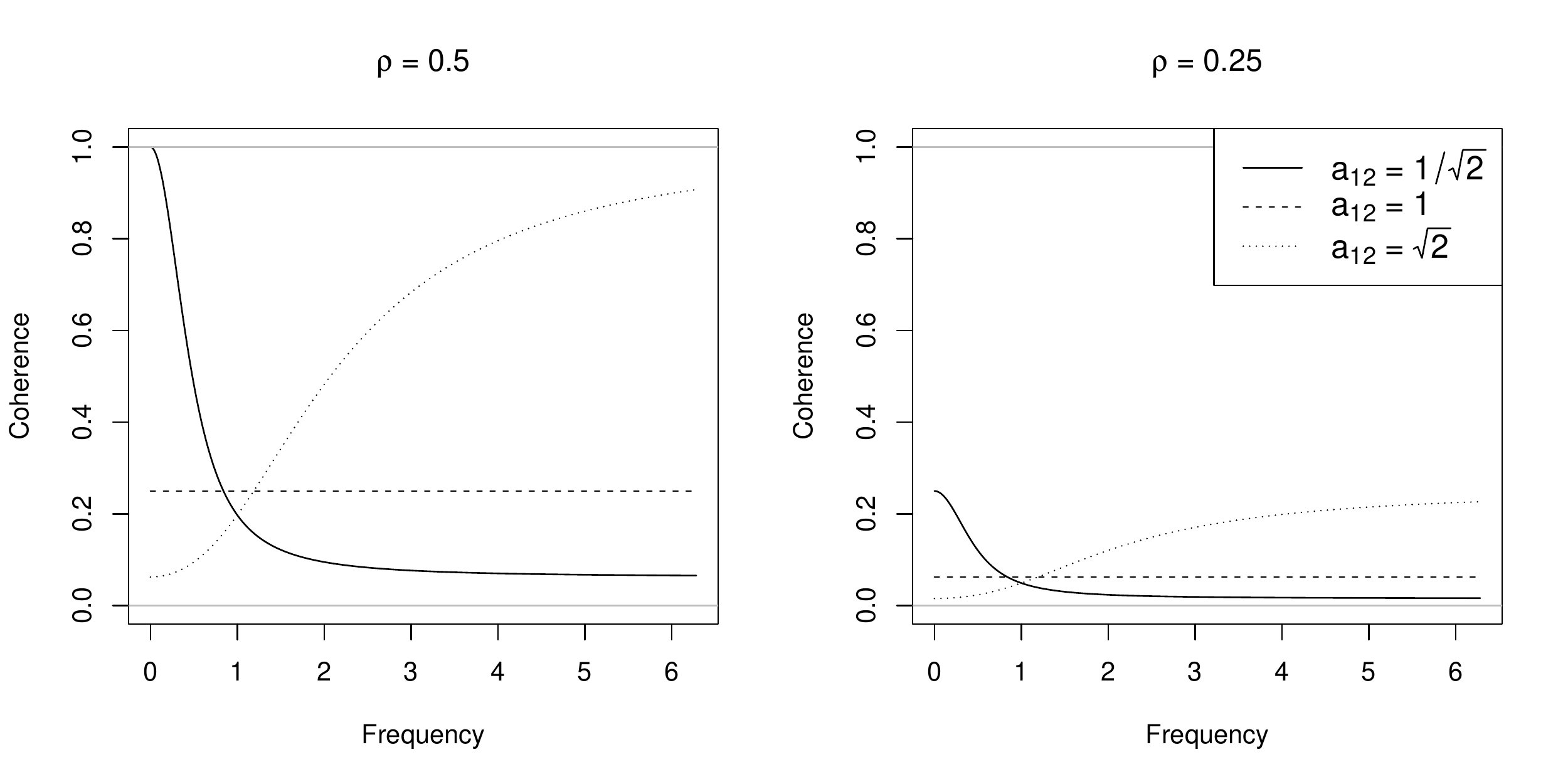}
  \caption{Coherence functions for various bivariate Mat\'erns with $\nu_{11} = 
  \nu_{12} = \nu_{22} = 1, a_{11} = a_{22} = 1$, varying $a_{12}$ and $\rho$.
  \label{fig:coh.mm.a}}
\end{figure}

The so-called parsimonious Mat\'ern model is defined by 
imposing common range parameters as well as $\nu_{12} = (\nu_1 + \nu_2)/2$ 
\citep{gneiting2010}.  This model has been empirically shown to produce inferior 
model fits to datasets as compared to more general versions of the multivariate 
Mat\'ern as well as other multivariate classes \citep{gneiting2010,apanasovich2012}.  
The coherence function for a bivariate parsimonious Mat\'ern model is constant, 
$\gamma(\bomega) = \rho$, which suggests an inflexible model for the spectral 
behavior of spatial processes.

We close this section with an empirical illustration of the implications of 
cross-covariance parameter choice on random field realizations and the associated 
low and high frequency behavior.  We simulate two bivariate Mat\'ern models  
on an equally-spaced grid of $256\times 256$ in $[0,8]^2$.  In both cases we 
low-pass and high-pass filter the resulting bivariate field.  The low-pass 
filter is a matrix of zeros except for a $3 \times 3$ grid of $1/9$, while the 
high-pass filter is similar with a $3 \times 3$ grid of $-1/9$ on the edge and 
$8/9$ in the center.  The effect of the filters is to remove high frequency 
behavior (low-pass filtering) or low frequency behavior (high-pass filtering).

Figure \ref{fig:filt.mm} shows low-pass filtered realizations of two 
bivariate Gaussian processes with bivariate Mat\'ern covariances.  
The top row is the case with equal ranges $a_{11} = a_{12} = a_{22} = 1$ but 
with a greater cross-smoothness $\nu_{11} = \nu_{22} = 0.5, \nu_{12} = 1$.  
According to Figure \ref{fig:coh.mm.nu}, we should expect the realizations 
to show similar low frequency behavior, but dissimilar high frequency behavior.  
Indeed, panels (a) and (b) show similar low frequency behavior, while 
the pairwise scatterplot of high-passed values, panel (c), suggests little 
correlation at high frequencies.  
Complementary, the second row is the case with equal smoothnesses, 
$\nu_{11} = \nu_{12} = \nu_{21} = 1$ but a greater cross-range parameter, 
$a_{11} = a_{22} = 1, a_{12} = \sqrt{2}$.  Figure \ref{fig:coh.mm.a} suggest 
we should expect less coherence at low frequencies while having greater 
correlation at high frequencies.  Again, these theoretical results are 
reinforced: panels (d) and (e) are not suggestive of strong low frequency 
coherence, while panel (f) exhibits positively correlated high 
frequency characteristics (and indeed has an empirical correlation coefficient of 
approximately $0.25$).

\begin{figure}[t]
  \centering
  \includegraphics[width=\linewidth]{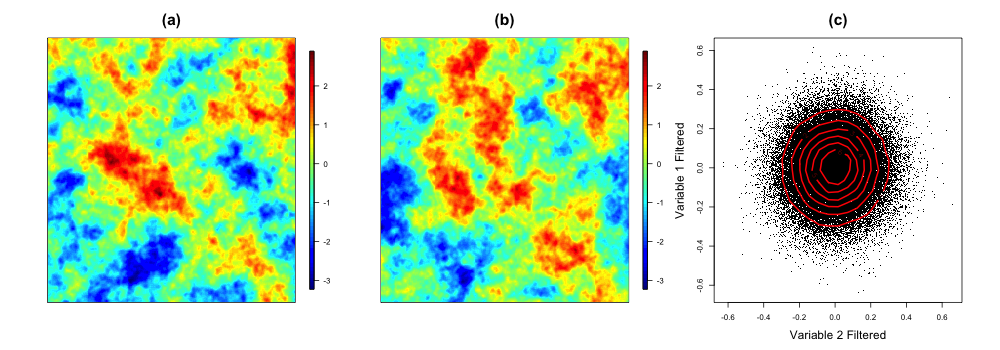}
  \includegraphics[width=\linewidth]{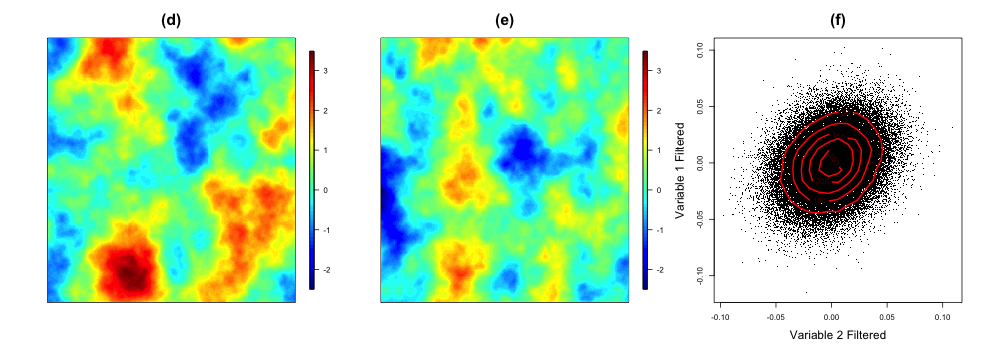}
  \caption{Low-pass filtered bivariate Mat\'ern with $a_{11} = a_{12} = 
  a_{22} = 1, \nu_{11} = \nu_{22} = 0.5, \nu_{12}=1$ and $\rho=0.5$ for the 
  (a) first process, (b) second process.  Panel (c) is the pairwise scatterplot of 
  high-pass filtered values of the two processes with contour levels for comparison. 
  Panels (d), (e) and (f) are analogous plots for a bivariate Mat\'ern simulation 
  with  $\nu_{11} = \nu_{12} = \nu_{22} = 1, a_{11} = a_{22} = 1, a_{12}=\sqrt{2}$ 
  and $\rho=0.5$.
  \label{fig:filt.mm}}
\end{figure}

\subsection{The Linear Model of Coregionalization}  \label{subsec:lmc}

The linear model of coregionalization (LMC) is a competing framework for multivariate 
modeling, and is built by decomposing a multivariate process as linear 
combinations of uncorrelated, univariate processes 
\citep{goulard1992,royle1999,wackernagel2003,schmidt2003}.  
In particular, we entertain the following version,
\begin{align}  \label{eq:LMC}
  \bZ(\bs) = 
  \left(
  \begin{array}{c}
    Z_1(\bs) \\
    Z_2(\bs)
  \end{array}
  \right)
  = 
  \left(
  \begin{array}{cc}
    b_{11} & b_{12} \\
    b_{21} & b_{22}
  \end{array}
  \right)
  \left(
  \begin{array}{c}
    W_1(\bs) \\
    W_2(\bs)
  \end{array}
  \right)
  = 
  \bB \bW(\bs)
\end{align}
The matrix $\bB$ is known as the coregionalization matrix, and controls the 
strength of dependencies on the latent uncorrelated processes $\bW$.  For 
the following, suppose $W_1(\bs)$ and $W_2(\bs)$ are uncorrelated processes 
with spectral densities $f_1(\bomega)$ and $f_2(\bomega)$, respectively.

Given the number of parameters in the LMC, it is often useful to impose 
restrictions on the coregionalization matrix $\bB$, such as setting 
$b_{11} = b_{22} = 1$ \citep{berrocal2010}.  
The following Lemma is the unsurprising result that the LMC yields multivariate 
processes that are exactly coherent when the coregionalization matrix has 
zero determinant.
\begin{lemma}  \label{lem:LMC.coh}
  In the linear model of coregionalization (\ref{eq:LMC}), if $b_{11} = b_{22} = 1$ 
  then the coherence function between $Z_1(\bs)$ and $Z_2(\bs)$ is unity if 
  and only if $b_{12} b_{21} = 1$.
\end{lemma}
\noindent Note that this result simply states that, under the LMC, two processes 
are exactly coherent when they differ only by a scalar multiplier.

Under the same working assumptions, $b_{11} = b_{22} = 1$, we have the gain 
function of $Z_2(\bs)$ on $Z_1(\bs)$ is 
\[
  G(\bomega) = \frac{b_{21} f_1(\bomega) + b_{12} f_2(\bomega)}
  {f_1(\bomega) + b_{12}^2 f_2(\bomega)}.
\]
If, as is common in using the LMC, we set $b_{12} = 0$, we have the gain function 
is simply $b_{21}$; that is, there is constant gain at all frequencies 
by the amount of coregionalization, $b_{21}$.  The 
complementary case where $b_{21} = 0$ yields the gain 
\[
  G(\bomega) = \frac{b_{12}f_2(\bomega)}{f_1(\bomega) + b_{12}^2 f_2(\bomega)},
\]
that is, the relative contribution of component $b_{12} W_2(\bs)$ to the combined 
spectrum of $Z_1(\bs)$.  As mentioned previously, if the latent processes have 
real-valued spectral densities, the phase function is exactly zero at all 
frequencies.

\section{Estimation of Spectra}  \label{sec:periodogram}

Suppose $\bZ(\bs)$ is a $p$-variate process that has been observed at a 
regular grid of points $\{\bs_i\}_{i=1}^N$, of marginal dimensions $n_i, 
i=1,\ldots,d$ where $N = \prod_{i=1}^d n_i$.  If grid spacing in the 
$i$th dimension is $\delta_i$, define $\delta = \prod_{i=1}^d \delta_i$.  
Then the spatial periodogram matrix with $(k,\ell)$th entry is defined as 
$\bI(\bomega) = (I_{k\ell}(\bomega))_{k,\ell=1}^p$ where 
\begin{align}  \label{eq:I}
  I_{k\ell}(\bomega) = \frac{\delta}{(2\pi)^p N} 
  \left( \sum_{k=1}^N Z_k(\bs_k) \exp(-i \bs_k\T \bomega) \right)
  \overline{\left( \sum_{k=1}^N Z_\ell(\bs_k) \exp(-i \bs_k\T \bomega) \right)}
\end{align}
is available at Fourier frequencies $\bomega = 2\pi\beff$ where 
$\beff=(f_1/(\delta_1 n_1),\ldots,f_d/(\delta_d n_d))\T$ for 
$f_i\in\{-\lfloor (n_i-1)/2\rfloor,\ldots,n_i-\lfloor n_i/2\rfloor\}$.  
Note that $I_{k\ell}(\bomega) = \overline{I_{\ell k}(\bomega)}$.  

Whereas in the time series literature it is natural to consider asymptotics 
as time $t\to\infty$, resulting in effectively uncorrelated blocks of a process, 
in the spatial realm there are two competing asymptotic frameworks.  
Increasing domain asymptotics is similar to the time series case where samples 
are taken on an ever-increasing domain in all axial directions, and typically 
asymptotic results here echo those in time series.  The 
complementary version is infill asymptotics (sometimes called fixed-domain 
asymptotics) in which the domain boundary is fixed and points are sampled at 
an ever finer resolution within the domain \citep{zhang2005}.  
Depending on the asymptotic framework under consideration, the large sample 
properties of the periodogram (\ref{eq:I}) change.  

Using infill asymptotics, \citet{lim2008} show that the raw multivariate 
periodogram can exhibit bias at low frequencies, and suggest prewhitening 
the process to overcome this inadequacy (in the univariate case \citet{stein1995} 
gives a simulated example where the bias is quite substantial).  However, 
under a mixture of infill and increasing domain asymptotics, \citet{fuentes2002} 
showed (for univariate processes) the analogous result to the time series case 
that the periodogram is asymptotically unbiased and is uncorrelated at 
differing Fourier frequencies.  Additionally, in this latter case it is not a 
consistent estimator, but must be smoothed to gain consistency.  

Analogous to the time series and univariate spatial field case, under certain 
assumptions the nonparametric periodogram (\ref{eq:I}) is asymptotically unbiased, 
and generates asymptotically 
uncorrelated random variables between distinct Fourier frequencies.  The 
following theorem illustrates this feature of the matrix-valued periodogram.
We use the same assumptions to \citet{fuentes2002}, generalized to the 
multivariate setting.
\begin{itemize}
  \item[A1]
    The true spectral densities $f_{k\ell}(\bomega)$ decay as $\|\bomega\|^{\tau}, 
    \tau>2$ as $\|\bomega\|\to\infty$, $\bomega\in\real^2$.
  \item[A2]
    The marginal and cross-covariances satisfy
    $\int \|\bh\| |C_{k\ell}(\bh)| {\rm d}\bh < \infty, \bh\in\real^2$.
  \item[A3]
    $\delta_i\to 0$, $n_i\to\infty$ and $\delta_i n_i\to\infty$ 
    for all $i,j=1,\ldots,d$ such that $n_i/n_j \to \lambda_{ij}>0$.
\end{itemize}
\begin{theorem}  \label{th:asymp}
  Under the assumptions A1-A3, we have 
  \begin{itemize}
    \item[(i)]
      $\E I_{k\ell}(\bomega) \to f_{k\ell}(\bomega)$, 
    \item[(ii)]
      $\Var I_{k\ell}(\bomega) = f_{k\ell}(\bomega)^2$ and 
    \item[(iii)]
      $\Cov(I_{k\ell}(\bomega_1),I_{k\ell}(\bomega_2))\to 0$ for $\bomega_1\not=
      \bomega_2$.
  \end{itemize}
\end{theorem}
The proof for Theorem \ref{th:asymp} follows directly from \citet{fuentes2002} 
and is not included here.

According to Theorem \ref{th:asymp}, the matrix-valued periodogram is not 
an asymptotically consistent estimator.  To produce a consistent estimator 
of the spectral density at a particular frequency $\bomega_0$, in practice we 
locally smooth adjacent periodogram values and appeal to property (iii) 
of Theorem \ref{th:asymp}.  In particular, the smoothed matrix-valued 
periodogram is 
\begin{align}  \label{eq:I.smooth}
  \tilde{I}_{k\ell}(\bomega_0) = \int_{\real^d} K_\lambda(\bomega - \bomega_0) 
  I_{k\ell}(\bomega) {\rm d}F_n(\bomega)
\end{align}
where $F_n(\bomega)$ is the empirical cumulative distribution function of 
Fourier frequencies $\{\bomega_i\}_{i=1}^N$.  Here, $K_\lambda$ is some kernel 
function with bandwidth $\lambda$, where, as we have it written, the same kernel 
is applied to each process.  Naturally, different kernels may be used for 
different processes if the scientific context calls for such an approach.

Note that we can't directly use the nonparametric periodogram fraction to 
estimate the coherence as $I_{k\ell}(\bomega)I_{\ell k}(\bomega) = 
I_{kk}(\bomega) I_{\ell\ell}(\bomega)$ at all Fourier frequencies.  Thus, we 
estimate the coherence functions by using the smoothed periodograms, 
\begin{align*}
  \hat{\gamma}_{k\ell}(\bomega)^2 = 
  \frac{|\tilde{I}_{k\ell}(\bomega)|^2}
  {\tilde{I}_{kk}(\bomega) \tilde{I}_{\ell\ell}(\bomega)}
\end{align*}
for $k,\ell=1,\ldots,p$.

\section{Illustrations}  \label{sec:example}

We examine two datasets from the atmospheric sciences, gridded reforecasts and 
reanalyses of sea level pressure and geopotential heights over the equatorial 
region.  
Reforecast data are produced retrospectively from a fixed version of a numerical 
weather prediction model, in this case the 2nd generation National 
Oceanic and Atmospheric Administration's (NOAA) Global Ensemble Forecast System 
Reforecast \citep{hamill2013}.  
Forecasts are generated at 3 hour increments from $0$ to $192$ hours, with the 
0h forecast being a reanalysis, that is, an estimate of the 
current state of the atmosphere.  For the data below, the control initial 
conditions were produced using a hybrid ensemble Kalman filter-variational 
analysis system \citep{hamill2011}.

\subsection{Sea Level Pressure}  \label{subsec:SLP}

The first dataset we consider is a set of reforecast sea level pressures (SLP) 
over the equatorial region.  Sea level pressures in this region are approximately 
stationary, and we compare forecast horizons in 24 hour increments from 
0h to 192h (8 days out).  The data consist of gridded reforecasts from the 
first 90 days of 2014 at $1^{\circ}$ increments over 360 longitude and 
47 latitude bands between $-23^\circ$ to $23^\circ$ defining the equatorial region.  

One approach to examining the quality of forecasts is the coherence between 
the forecast with the corresponding reanalysis.  For example, we might compare 
the 24h forecast of SLP generated on January 1, 2014 to the 0h reanalysis 
generated on January 2, 2014.  It is well known that forecast skill 
decays with horizon, and we expect the short-term forecasts to share higher 
coherence with the reanalyses than the long-term forecasts.  

We begin by standardizing each analysis and forecast horizon grid cell by 
subtracting the temporal average and diving by empirical standard standard deviation 
to produce forecast anomalies.  Denote these anomalies by 
$Z_k(\bs,d)$ for forecast horizons $k=0,1,2,\ldots,8$ corresponding to 
forecast horizons $0,24,\ldots,192$ hours, spatial locations 
$\bs\in{\cal D}\subset\real^2$ in the equatorial region ${\cal D}$ on days 
$d=1,\ldots,90$, i.e., the first 90 days of 2014.  

Each day's marginal process empirical periodogram (\ref{eq:I}) is calculated 
for all forecast horizons $k$, yielding $\{I_{kk}(\bomega,d)\}$.  The 
smoothed periodogram is a convolution with a simple low-pass filter, a matrix 
of zeros with a $3 \times 3$ constant block of $1/9$.  Interest focuses on 
comparing various forecast horizons with the reanalysis at $k=0$, so we 
calculate empirical cross-periodograms $\{I_{0k}(\bomega,d)\}$ for all available 
days $d$ allowing for forecast validation (e.g., the $k=1$, 24h horizon, has 89 
available days, $d=2,\ldots,90$).  The cross-periodograms are smoothed using the same 
low-pass filter as the marginals.  If $\tilde{I}_{k\ell}(\bomega,d)$ denotes 
the smoothed cross-periodograms, we estimate the squared coherence function as 
\[
  \hat{\gamma}_{0k}(\bomega)^2 = \frac{1}{90 - k}\sum_{d=1+k}^{90} 
  \frac{|\tilde{I}_{0k}(\bomega,d)|^2}{\tilde{I}_{00}(\bomega,d) 
  \tilde{I}_{kk}(\bomega,d-k)},
\]
for $k=1,\ldots,8$, that is, the average over all available daily 
smoothed cross-periodograms.

\begin{figure}[t]
  \centering
  \includegraphics[width=\linewidth]{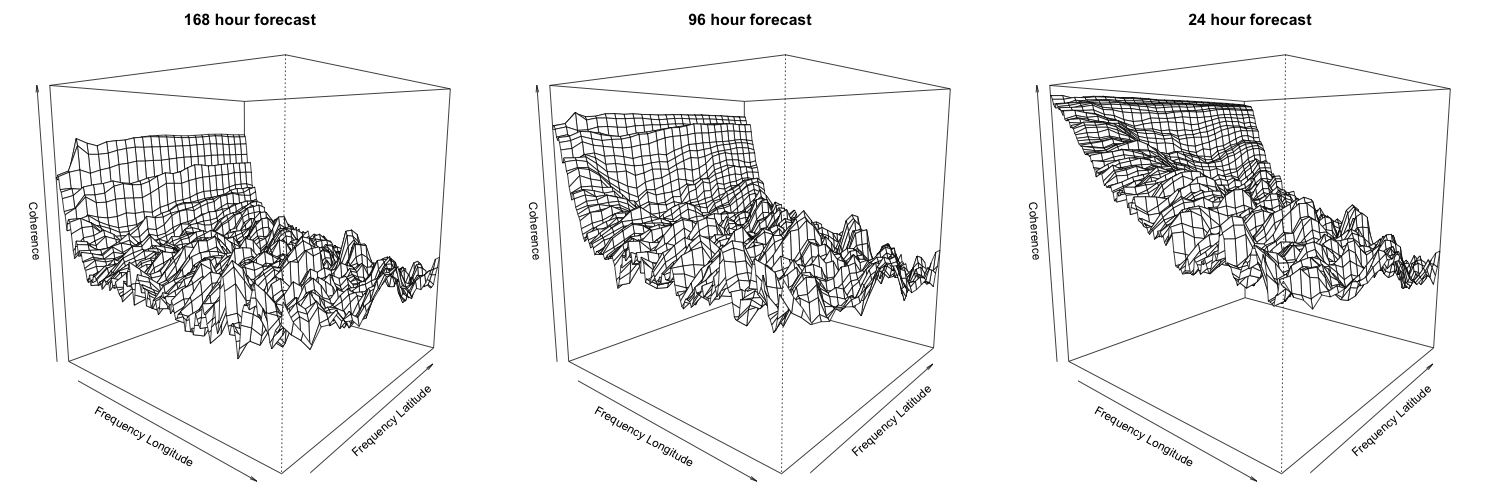}
  \caption{Estimated absolute coherence functions for the GEFS sea level pressure 
  reforecast data, comparing the 168, 96 and 24 hour forecast horizons with 
  the zero hour analysis.  The vertical axis spans $[0,1]$.
  \label{fig:GEFS.pres.coh}}
\end{figure}

Figure \ref{fig:GEFS.pres.coh} shows estimated absolute coherence functions 
for horizons 168, 92 and 24h with the 0h analysis.  Even at long lead lead times 
there is substantial coherence, which increases by a substantial margin 
at very low longitudinal frequencies.  For any given longitude frequency band, 
the coherence appears to be relatively constant across latitudes, which 
is sensible given that there appears to be 
greater variability in the equatorial direction 
than in the north-south direction for sea level pressure in this region.  
As the forecast horizon decreases the coherence begins building between 
low-to-mid frequency bands in the latitudinal direction, suggesting 
that the statistical characteristics of short term forecasts are more similar 
to observed sea level pressure than the longer term forecasts.  
However, note that at the highest frequencies there is not a substantial 
improvement in forecast skill, bordering on no improvement, which 
suggests that small scale events are difficult to forecast even at one day out.

To quantify the differences in forecast horizon skill, we estimate a set of 
bivariate Mat\'ern models.  As each forecast and analysis arises from the 
same physical model, we assume that the marginal spectra follow the same 
statistical behavior, that is, we suppose the marginal spectral densities at equal 
at all horizons, $f_{00}(\bomega) = f_{kk}(\bomega)$.  Suggested by 
Figure \ref{fig:GEFS.pres.coh} and exploratory analysis for the marginal 
spectra, we additionally suppose the spectrum is constant across latitude 
frequencies.  A Gaussianity assumption does not appear to be justifiable, 
based on empirical Q-Q plots.  Thus, we follow the suggestion of \citet{fuentes2002} 
and estimate marginal Mat\'ern parameters by minimizing 
the squared difference between theoretical log spectral density and 
log average periodogram, having averaged over all days, forecast horizons and 
latitude bands.  The resulting least squares estimates are $a = 0.074$ and 
$\nu = 0.94$.  Cross-covariance parameters are estimated by minimizing least 
squares distance to the average empirical coherence functions, estimated by 
averaging over days and latitude bands.  

\begin{table}[tp] 

\caption{Cross-covariance Mat\'ern family parameter estimates for GEFS 
  sea level pressure reforecast data.
  \label{tab:GEFS.pres.matern}}

\bigskip

\centering
\begin{tabular}{|c|cccccccc|}
  \hline \hline
  Forecast horizon & 24 & 48 & 72 & 96 & 120 & 144 & 168 & 192 \\
  \hline
  $\rho$ & 0.96 & 0.94 & 0.91 & 0.88 & 0.85 & 0.81 & 0.76 & 0.73 \\
  $a_{12}$ & 0.079 & 0.076 & 0.075 & 0.073 & 0.071 & 0.070 & 0.068 & 0.067 \\
  $\nu_{12}$ & 1.06 & 1.06 & 1.05 & 1.04 & 1.03 & 1.02 & 1.01 & 1.01 \\
  \hline
\end{tabular}

\end{table}

Table \ref{tab:GEFS.pres.matern} contains the cross-covariance parameter 
estimates corresponding to all forecast horizons.  As forecast horizon 
increases, all parameters decay; in fact, the decay is almost exactly 
linear for each variable beyond the 24h horizon.  Fitting a linear model 
to the parameters as a function of forecast horizon suggests the cross-covariance 
smoothness will decay to the marginal value $\nu = 0.94$ after approximately 
16 days, whereas the cross-covariance scale meets the marginal value between 
3-4 days.  Although these ideas can be used to generate scientific hypotheses, 
there is substantial extrapolation for $nu_{12}$ that strongly depends on a 
linear assumption.

A word of caution is in order; the estimates in Table 
\ref{tab:GEFS.pres.matern} do not always imply a valid multivariate covariance 
structure.  However, on any given grid the estimated parameters may yield 
a valid model, it just cannot be guaranteed for all grids.  One possibility is 
that the bivariate Mat\'ern is not sufficiently flexible to describe the 
stochastic structure of these fields, which is a call for further research 
in this area.

\subsection{Geopotential Height}  \label{subsec:geoheight}

Our second example is on the same spatial domain, but whose values are 
geopotential heights.  Geopotential height is the height (in meters) above 
sea level at which the atmospheric pressure is a certain level.  
In the atmospheric sciences, it is common to examine geopotential heights 
as indicators of climatic regimes, for instance \citet{knapp1996} examines 
the relationship between heights and temperature anomalies over a portion of 
the United States.  

Three of the most common geopotential height maps are the 850hPa, 500hPa and 300hPa 
maps.  The first, 850hPa, approximately defines the planetary boundary layer, 
which is the lowest level of the atmosphere that interacts with the surface 
of the Earth (note 1000hPa is approximately sea level).  The 300hPa level 
is at the core of the jet stream, while the 500hPa approximately divides 
the atmosphere in half, and whose anomalies are used in part to assess climatological 
temperature variations.  The vertical structure of geopotential heights 
is a focus of some interest within atmospheric sciences \citep{blackmon1979}.

We examine geopotential height reanalysis anomalies $Z_k(\bs,d)$ for 
$k=1,2,3$ representing 
the 850hPa, 500hPa and 300hPa pressure levels on days $d=1,\ldots,181$, the first 
6 months of 2014.  The anomalies are differences between the reanalysis 
height and a time-varying Nadaraya-Watson kernel smoothed estimate of the 
mean with a bandwidth of 5 days.  Experiments suggest the results below are 
qualitatively robust against choices of the bandwidth and smoothing kernel.  

Similar to the previous section, we smooth the marginal process periodogram 
(\ref{eq:I}) using a low-pass filter, and calculate smoothed empirical 
cross-periodograms yielding $\{\tilde{I}_{ij}(\bomega,d)\}_{i,j=1}^3$.  
Then the squared coherence is estimated as the arithmetic average of each 
day's empirical squared coherence estimate,
\[
  \hat{\gamma}_{ij}(\bomega)^2 = \frac{1}{181}\sum_{d=1}^{181} 
  \frac{|\tilde{I}_{ij}(\bomega,d)|^2}{\tilde{I}_{ii}(\bomega,d) 
  \tilde{I}_{jj}(\bomega,d)}.
\]

\begin{figure}[t]
  \centering
  \includegraphics[width=\linewidth]{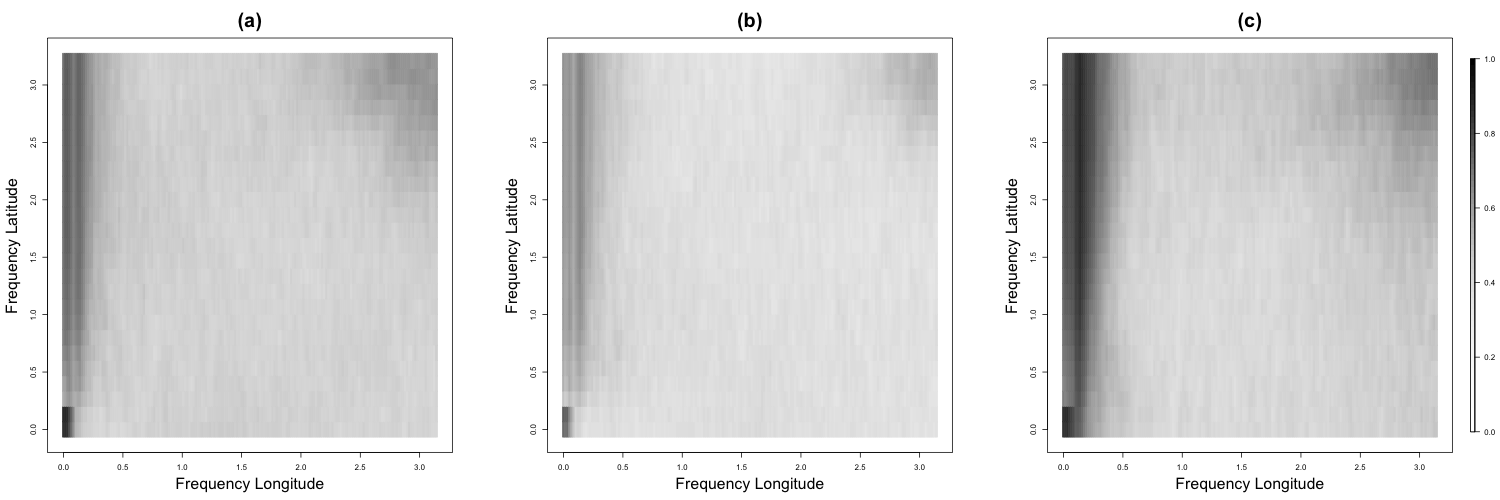}
  \caption{Estimated absolute coherence functions for the GEFS geopotential 
  height data between (a) 850hPa and 500hPa, (b) 850hPa and 300hPa and 
  (c) 500hPa and 300hPa.
  \label{fig:GEFS.geoheight.coh}}
\end{figure}


Figure \ref{fig:GEFS.geoheight.coh} contains the three estimated pairwise absolute 
coherence functions.  There is high coherence between the lower 
pressure levels at low frequencies, and some evidence of moderate coherence 
between all levels at low frequencies.  
We also note some strikingly different behavior than 
for the sea level pressure example.  First, there is an apparent ridge 
in coherence at low frequencies (approximately $2\pi9/360$) which 
may be indicative of equatorial planetary waves \citep{wang1996,xie1996,kiladis2009}.  
Planetary waves can play crucial roles in the formation of tropical cyclones 
\citep{molinari2007}.  
Additionally, note that there is high coherence at the highest Fourier frequencies 
for all coherence functions (capping out at approximately 
$0.62, 0.54,$ and $0.70$).  This is evidence of a nonseparable relationship 
in the frequency domain, and we are unaware of any current multivariate 
models that can adequately capture such behavior.  One possible explanation 
for this high coherence at high frequencies is artifacts in the data assimilation 
scheme, in particular aberrant observational data leading to unusually 
large anomalies in geopotential height.  Indeed, variables such as sea 
level pressure are well constrained by a wealth of observational data, while 
geopotential heights are less constrained, usually being observed by sparsely 
released weather balloons.

\begin{figure}[t]
  \centering
  \includegraphics[width=\linewidth]{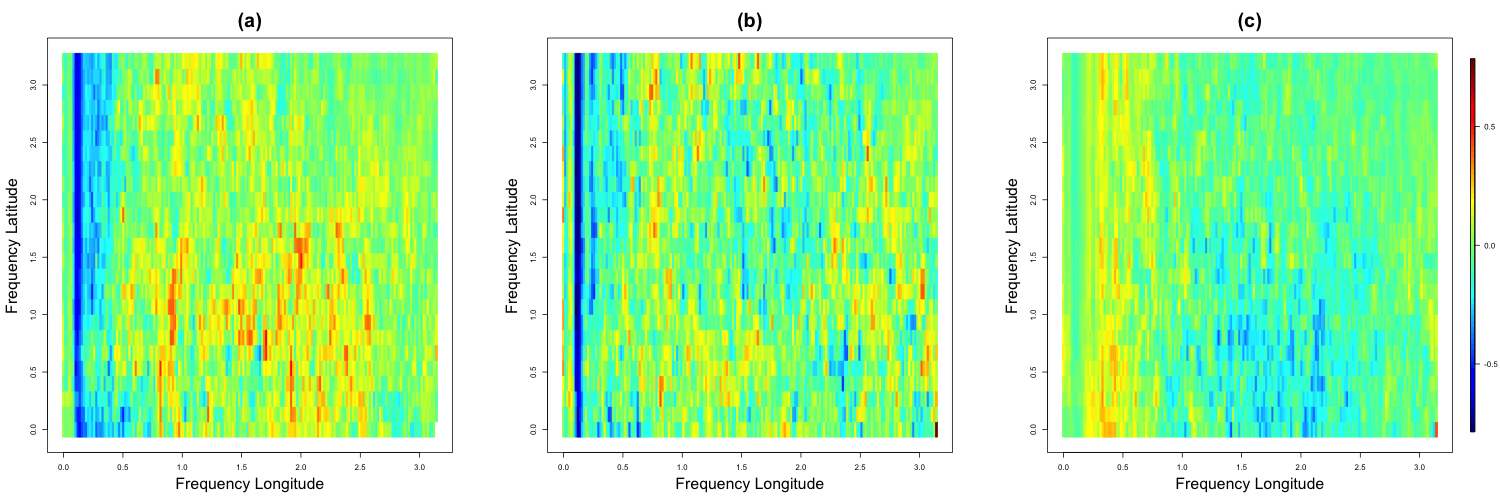}
  \caption{Estimated phase functions for the GEFS geopotential 
  height data between (a) 850hPa and 500hPa, (b) 850hPa and 300hPa and 
  (c) 500hPa and 300hPa.
  \label{fig:GEFS.geoheight.ph}}
\end{figure}

Figure \ref{fig:GEFS.geoheight.ph} shows pairwise plots for each pair of geopotential 
height anomalies.  In particular there is strong evidence of a phase shift 
at a low frequency band between the 850hPa and both lower pressure heights.  
These frequencies indicate wavelengths of approximately $4000-6000$km, which 
is a typical wavelength for planetary equatorial waves, or Rossby waves 
\citep{wang1996,xie1996,kiladis2009}.  
On the other hand, there is not substantial evidence of phase shift 
between the pairs of heights at other frequencies.  Most extant multivariate 
models utilize real-valued cross-spectral densities, and thus are insufficiently 
flexible to capture this type of phase shifted behavior at specific 
spectra.  

\section{Discussion}  \label{sec:discussion}

The notion of coherence, phase and gain are common in the time series literature, 
but have been yet unexplored for multivariate spatial processes.  We casted 
these functions for multidimensional processes.  The coherence between two 
variables can be interpreted as a measure of linear relationship at 
particular frequency bands, resulting in a complementary framework for 
comparing processes than the usual cross-covariance function.  
Phase and gain also yield straightforward interpretations as a physical space-shift 
and relative amplitude of frequency dependence when comparing two processes.  
We developed these ideas for stationary processes, and future research 
may be directed toward the analogous cases for nonstationary processes, 
perhaps extending the work of \citet{fuentes2002}.

Coherence, phase and gain can be estimated using smoothed cross-periodograms, 
and in our examples we showed that, as exploratory tools, these functions 
can be very useful in detecting structure that may not be readily 
captured using extant multivariate models.  We additionally illustrated that 
the coherence function gives a natural interpretation to the multivariate 
Mat\'ern cross-covariance parameters that have otherwise been uninterpretable, 
lending insight into an outstanding problem.  

A number of future research directions may be considered, including the 
adaptation of coherence to multivariate space-time processes.  This work 
can also be seen as a call to develop more flexible multivariate models, 
perhaps working directly in the spectral domain rather than the covariance 
domain, where most previous work has fallen, echoing the call of 
\citet{simpson2015}.

\section*{Appendix}

The following Lemma is useful in proving some results of the manuscript.
\begin{lemma}  \label{lem:intKZ.relation}
  Suppose $Z_1(\bs)$ is a stationary processes on $\real^d$ with covariance 
  function $C_1(\bh)$ having spectral density $f_1(\bomega)$ 
  and $Z_2(\bs) = \int K(\bs - \bu) Z_1(\bu) {\rm d}\bu$ 
  where $K$ is continuous, symmetric and square integrable with Fourier transform 
  $f_K(\bomega)$.  Then $Z_2(\bs)$ has covariance function 
  $C_2(\bh) = \int\int K(\bu + \bv - \bh) K(\bv) C_1(\bu) {\rm d}\bu {\rm d}\bv$ 
  with associated spectral density $f_2(\bomega) = f_1(\bomega) f_K(\bomega)^2$.  
  Additionally, the cross-covariance function between $Z_1$ and $Z_2$ is 
  $C_{12}(\bh) = \int K(\bu - \bh) C_1(\bu) {\rm d}\bu$ with spectral density 
  $f_{12}(\bomega) = f_1(\bomega) f_K(\bomega)$.
\end{lemma}
The proof of this Lemma involves straightforward calculations involving 
convolutions and is not included here.

We recall the spectral representation for a stationary vector-valued process 
$\bZ(\bs)\in\real^p, \bs\in\real^d$ with matrix-valued covariance function 
$\bC(\bh)$ having spectral measures $F_{ij}, i,j=1,\ldots,p$ defined on the 
Borel $\sigma$-algebra ${\cal B}$ on $\real^d$.  
There is a set of complex random measures $\bM=(M_1,\ldots,M_p)$ 
on ${\cal B}$ such that if $B,B_1,B_2\in{\cal B}$ are disjoint, $\E M_{i}(B) = 0,  
\E( M_{i}(B) \overline{M_{j}(B)} ) = F_{ij}(B)$ and 
$\E( M_{i}(B_1) \overline{M_{j}(B_2)} ) = 0$ for $i,j=1,\ldots,p$.  
Then $\bZ(\bs)$ has the spectral representation
\[
  \bZ(\bs) = \int \exp(i\bomega\T\bs) {\rm d}\bM(\bomega),
\]
see \citet{gihman1974} for details.  
If all $F_{ij}$ admit associated spectral densities $f_{ij}$, then in 
shorthand we write $\E({\rm d}M_i(\bomega)\overline{{\rm d}M_j(\bomega)}) 
= f_{ij}(\bomega){\rm d}\bomega$.

\begin{proof}[Proof of Theorem \ref{th:prediction}]
  The spectral representation implies 
  \[
    Z_i(\bs) = \int \exp(i\bomega\T\bs) {\rm d}M_i(\bomega),
  \]
  for complex-valued random measures $M_i$, $i=1,2$.  
  Then if $K$ has Fourier transform $F_K$, 
  \begin{align*}
    \int K(\bu - \bs) Z_2(\bu) {\rm d}\bu &= 
    \int\int K(\bu - \bs) \exp(i\bomega\T\bu) {\rm d}M_2(\bomega) {\rm d}\bu \\
    &= \int \exp(i\bomega\T\bs) F_K(\bomega) {\rm d}M_2(\bomega)
  \end{align*}
  by a change of variables.  Then, using that $f_{ii}(\bomega){\rm d}\bomega = 
  \E|{\rm d}M_i(\bomega)|^2$ and 
  \[
    \E\left( \int g(\bomega){\rm d}M_i(\bomega) \overline{\int h(\bomega) 
    {\rm d}M_j(\bomega)}\right) = 
    \int g(\bomega) \overline{h(\bomega)} f_{ij}(\bomega) {\rm d}\bomega
  \]
  we have 
  \begin{align*}
    \E\left| Z_1(\bs) - \int K(\bu - \bs) Z_2(\bu) {\rm d}\bu \right|^2 
    &= \int\bigg( f_{11}(\bomega)  - f_{12}(\bomega) F_K(\bomega) - 
    f_{21}(\bomega)\overline{F_K(\bomega)} + \\
    &\qquad F_K(\bomega)\overline{F_K(\bomega)}f_{22}(\bomega)\bigg) {\rm d}\bomega \\
    &= \int \E\left| {\rm d}M_1(\bomega) - F_K(\bomega) {\rm d}M_2(\bomega)\right|^2.
  \end{align*}
  The integrand is minimized for each $\bomega$ if 
  \[
    F_K(\bomega) = \frac{\E({\rm d}M_1(\bomega) \overline{{\rm d}M_2(\bomega)})}
    {\E|{\rm d}M_2(\bomega)|^2}
    = \frac{f_{12}(\bomega)}{f_{22}(\bomega)}.
  \]
  That the density of $\int K(\bu - \bs) Z_2(\bu) {\rm d}\bu$ is 
  $|f_{12}(\bomega)|^2 / f_{22}(\bomega)$ now follows by the convolution theorem 
  for Fourier transforms.
\end{proof}

\begin{proof}[Proof of Proposition \ref{prop:c.conv}]
  If $f_{i}(\bomega)$ is the Fourier transform of $c_i, i=1,2$, the result immediately 
  follows as the spectral density of $C_{ij}(\bh)$ is $f_i(\bomega)f_j(\bomega)$.
\end{proof}

\begin{proof}[Proof of Proposition \ref{prop:z.conv}]
  This result follows directly from Lemma \ref{lem:intKZ.relation}.
\end{proof}

\begin{proof}[Proof of Proposition \ref{prop:k.conv}]
  This result follows from Lemma \ref{lem:intKZ.relation} and that the 
  spectral density for white noise is constant over all frequencies.
\end{proof}

\section*{Acknowledgements}

The author thanks Michael Scheuerer for many helpful discussions during the 
development of this research.  
This research was supported by National Science Foundation grants 
DMS-1417724 
and DMS-1406536.


\end{document}